\newcommand{\chm}[1]{{\textcolor{black}{#1}}}

\documentclass[american,english]{article}
\usepackage[T1]{fontenc}
\usepackage[latin9]{inputenc}
\usepackage{geometry}
\usepackage{amsmath}
\usepackage{amsthm}
\usepackage{amssymb}
\usepackage{siunitx}

\usepackage{fullpage}
\usepackage{amssymb}
\usepackage{amsmath}
\newtheorem{remark}{Remark}
\newtheorem{proposition}{Proposition}
\usepackage{tikz}
	\tikzset{
    	every picture/.append style={
		    line join = round,
		    line cap = round,
		},
	}
    \usetikzlibrary{%
        decorations.pathreplacing,%
        decorations.pathmorphing,%
        arrows,%
    }
    \tikzset{
        interface/.style={
            postaction={
                draw,
                thin,
                decorate,
                decoration={
                    border,
                    angle=-45,
                    amplitude=0.3cm,
                    segment length=2mm,
                },
            },
        },
        spring/.style={
            decorate,
            decoration={
                coil,
                aspect=0.3,
                segment length=1ex,
                amplitude=1ex,
                pre length=.5em,
                post length=.5em,
            },
        },
    }
	\usepackage{pgfplots}
		\pgfplotsset{compat=1.13}  
		\usepgfplotslibrary{units}
		\usepgfplotslibrary{colorbrewer}
		\pgfplotsset{
    		every axis/.append style={
    			scale only axis,
    			width = \textwidth - 4cm,
    			height = 3.5cm,
			},
		}

\usepackage{siunitx}
\usepackage{cool}
\usepackage{interval}
\usepackage{nicefrac}
\usepackage{varwidth}
\usepackage{linegoal}
\usepackage{calc}
\usepackage[blocks]{authblk}
\usepackage{enumitem}
	\setlist[itemize]{noitemsep, topsep=0pt}

	\usepackage{standalone}
	\standaloneconfig{mode=image|tex}

\usepackage{pgfplotstable}
	\usepackage{booktabs}
\usepackage{gitinfo}
\usepackage{tikz,pgfplots}
\usepackage{pgfplotstable}

\title{Mathematical formulation of a dynamical system with dry friction subjected to external forces}

\author{A. Bensoussan$^1$, A. Brouste$^2$, F. B. Cartiaux$^3$, C. Mathey$^4$ and L. Mertz$^5$}
\affil{
$^1$Jindal School of Management, University of Texas at Dallas, Richardson, USA; School of Data Science, Hong-Kong City University, Hong Kong,\\
$^2$Laboratoire Manceau de Math\'ematiques, Le Mans Universit\'e, Le Mans, France,\\
$^3$OSMOS Group,\\
$^4$Department of Mathematics, City University of Hong Kong, Kowloon Tong, Hong Kong,\\
$^5$ECNU-NYU, Institute of Mathematical Sciences, NYU Shanghai, Shanghai, China}
\begin{document}
\maketitle
\begin{abstract}
We consider the response of a one-dimensional system with  friction.
S.W. Shaw (Journal of Sound and Vibration, 1986) introduced the set up of different coefficients for the static and dynamic phases (also called stick and slip phases). He constructs a step by step solution, corresponding to an harmonic forcing. 
In this paper, we show that the theory of variational inequalities (V.I.) provides an elegant and synthetic approach to obtain the existence and uniqueness of the solution, avoiding the step by step construction. We then apply the theory to a real structure with real data  and show that the model is quite accurate. In our case, the forcing motion comes from dilatation, due to temperature. 
\end{abstract}

\section{Introduction}
The study of one dimensional models of \textit{dry friction} (also called \textit{Coulomb friction}) has received a significant interest over the last decades \cite{DENHAR1930,MR0852197,PFD1985,YEH1966,LCW1982,LCW1983}.
A typical  example is  a solid lying on a motionless surface.
In response to forces, the dynamics of the solid has two separate phases.
One is \textit{dynamic} (also called slip) in which the solid moves and a dynamic friction force opposes the motion. 
The other phase is \textit{static} (also called stick) in which the solid remains motionless while being subjected to a static friction force that is necessary for equilibrium. The physics of dry friction is presented in \cite{POP2010}. The static phase plays an important role on the prediction of the system behavior (e.g. failure) and thus it has to be carefully analyzed.

\noindent The case when external forces are harmonic has been investigated with particular attention. 
Two  classes of techniques have been developed in the literature.  
In dimension one or two, \textit{exact methods} have been considered. 
Den Hartog \cite{DENHAR1930} has proposed  an exact solution of the dynamic phase for a single degree of freedom system where the coefficients of static and dynamic friction are equal. Shaw \cite{MR0852197} extended Den Hartog's results with a different approach to the case where the coefficients of static and dynamic friction are not identical and provided a stability analysis of the periodic motion. Yeh \cite{YEH1966} extended the method to two degrees of freedom with one dry friction damper. 
In higher dimensions, an approximation method called  \textit{the incremental harmonic balance (IHB) method}, has been developed by Lau, Cheung \& Wu \cite{LCW1982,LCW1983} and Pierre, Ferri \& Dowel \cite{PFD1985}. This method is efficient   when dealing with many degrees of freedom and many dry friction dampers. The IHB method does provide good results on amplitude and phase for a broad class of stick/slip motions but it does not give detailed information about stick/slip phases.

\noindent 
In this paper, we study a one-dimensional problem, with general forcing term.
Denoting the actual displacement of the oscillator by $x$, we consider the initial value problem
\begin{equation}
\label{eq:friction1}
m \ddot x(t) + \mathbb{F}(\dot x(t)) = b(x(t),t), \: t >0
\end{equation}
with $m>0$ is a constant, the initial displacement and velocity 
\begin{equation}
\label{eq:friction2}
x(0) = x_0, \quad \dot x (0) = 0.
\end{equation}
In the right hand side of Equation \eqref{eq:friction1}, $b(x(t),t)$ represents  the  forces that are not related to dry friction; here $b$ is a locally Lipschitz function with at most linear growth. 
In our application, we will take 
\begin{equation}
\label{eq:engapp}
b(x,t) \triangleq K (\beta T(t) - x)
\end{equation}
with $K, \beta$  constants and $T$ is a temporal function which describes the evolution of the temperature.
The function $\mathbb{F}(\dot x(t))$ is not easy to describe. In a static phase $\dot x(t)=0$ on an interval, which implies $\ddot x(t)=0$ on the same interval and therefore from \eqref{eq:friction1} 
\begin{equation}
\label{eq:friction200}
\mathbb{F}(\dot x(t)) = b(x(t),t)
\end{equation}
However,  this is possible only when $|b(x(t),t)| \leq f_s$,  where $f_s$ is the static friction coefficient.
In a  dynamic phase $\dot x(t)\neq 0$ on an interval, although it can vanish at isolated points. We cannot simply use equation \eqref{eq:friction1} to obtain $\mathbb{F}(\dot x(t))$. We need an additional information, provided by Coulomb's law. We write, formally 
\begin{equation}
\label{eq:friction201}
\mathbb{F}(\dot x(t)) = f_d\mbox{sign}(\dot x(t)).
\end{equation}
Since we expect $\dot x(t)$  to vanish only at isolated points, $\mbox{sign}(\dot x(t))$  is defined  almost everywhere, and therefore, we can use this expression in the second order differential equation \eqref{eq:friction1}, with an equality valid a.e. instead of for any $t$. The difficulty is that we cannot write the equation a priori. The step by step construction of the solution is  a way to get out of this chicken and egg effect. A more elegant way is to use variational inequalities (V.I.). A V.I. is the following mathematical problem 
\begin{equation}
\label{eq:vi}
\tag{$\mathcal{VI}$}
\forall t > 0,
\: \forall \varphi \in \mathbb{R}, 
\: (b(x(t),t) - m \ddot x (t)) (\varphi - \dot x(t)) + f_d |\dot x(t)| \leq f_d |\varphi|.
\end{equation}
If we apply the  V.I  in   a static phase, when $\dot x(t)=0$ on an interval, we get immediately that $|b(x(t),t)|\leq f_d$.  This cannot apply to  the case considered by Shaw in which in a static phase $|b(x(t),t)|\leq f_s$  with a coefficient $f_s>f_d$. We will see how to solve this difficulty in the next section.

\begin{remark}
We can change  $b(x(t),t)$  into $b(x(t),\dot x(t),t)$ in the right hand side of Equation \eqref{eq:friction1},  provided the dependence in $\dot x(t)$   is smooth. For instance, we have in mind forces of the form $b(x(t),\dot x(t),t)=b(x(t),t) - \alpha \dot x(t)$. To simplify we shall omit this situation. 
 
\end{remark}
\noindent  In this paper, we propose a two-phase model for Equations \eqref{eq:friction1}-\eqref{eq:friction2} with $f_d \leq f_s $ and obtain  a  solution which is $C^1$.   The two phases are called static and dynamic phases. The dynamic phase is captured by a V.I. In the case $f_d = f_s$, the VI   captures both phases.

\section{Model description and mathematical theory}
\subsection{Description of dry friction: static and dynamic phases}
In presence of dry friction, the physical description of the dynamic and static phases is as follows. 
\begin{itemize}
\item
The phase of $x(t)$ at time $t$ is \textbf{static} when 
$$
\dot x (t) = 0 \: \mbox{and} \: |b(x(t),t) | \leq f_s.
$$
It is important to emphasize that a static phase corresponds to $\dot x(t) = 0$ on an time interval of positive Lebesgue measure and thus $\ddot x(t) = 0$ on the same interval. In this case, the friction force takes the value
$$
\mathbb{F}(\dot x(t)) = b(x(t),t)
$$
that is necessary for equilibrium.
\item
Otherwise, the phase of $x(t)$ at time $t$ is \textbf{dynamic} when
$$
\dot x (t) = 0 \: \mbox{and} \: |b(x(t),t) | > f_s 
\: \mbox{or} \:
\dot x (t) \neq 0.
$$
When $\dot x (t) \neq 0$, the friction force takes the value
$$
\mathbb{F}(\dot x(t)) =  \textup{sign} (\dot x(t)) f_d.
$$
It is important to emphasize that $\dot x (t) = 0$ \mbox{and} $|b(x(t),t) | > f_s$ happen on a negligible (isolated points) time set. 
\end{itemize}
The transition from a static phase to a dynamic phase occurs as soon as 
$$
|b(x(t),t)| > f_s.
$$ 
Conversely, from a dynamic phase the system enters into a static phase as soon as 
$$
\dot x (t) = 0 \: \mbox{and} \: |b(x(t),t)| \leq f_s.
$$ 
Below, we make precise the mathematical formulation and obtain a  solution which  is $C^1$. 

\subsection{Mathematical theory: a two  phase  model in which the dynamic phase is modeled by a variational inequality}
We want to define rigorously a function $x(t)$, which is $C^1$ 
and satisfies \eqref{eq:friction1}-\eqref{eq:friction2}. To fix ideas, we consider initial conditions 
$\tau_0 = 0$, $x(\tau_0) = x_0$, $\dot x(\tau_0) = 0$ in a static phase with
$$
|b(x_0,\tau_0)| \leq f_s.
$$ 
The first transition towards a dynamic phase occurs at time $\tau_{\frac{1}{2}}$ defined as 
$$
\tau_{\frac{1}{2}} 
\triangleq 
\inf \left \{ t \geq \tau_0, \: |b(x_0,t)| > f_s \right \} 
\footnote{The function $t \to b(x_0,t)$ being continuous, it is worth mentioning that $|b(x_0,\tau_{\frac{1}{2}})| = f_s$ \underline{but} whenever $\epsilon>0$ is appropriately small then $|b(x_0,\tau_{\frac{1}{2}} + \epsilon)| > f_s$.}
$$
If $\tau_0 < \tau_ {\frac {1} {2}}$ then on the interval $[\tau_0, \tau_{\frac{1}{2}})$ we define
$$
x(t) \triangleq x_0, \: \dot x (t) \triangleq 0.
$$
If $|b (x_0, t) |$ never goes beyond $f_s$ then $\tau _ {\frac {1} {2}} = \infty $ and $x$ remains forever in a static phase.
If $\tau_0 = \tau_ {\frac {1} {2}}$ there is no static (or stick) phase and thus the interval $[\tau_0, \tau_ {\frac {1} {2}})$ is void.
Provided that $\tau _ {\frac {1} {2}}$ is finite, we can define the time of return in static phase $\tau_1$ as 
$$
\tau_1 
\triangleq 
\inf \left \{ t \: >  \: \tau_{\frac{1}{2}} , \: \dot x (t) = 0 \: \mbox{and} \:  |b(x(t),t)| \leq f_s \right \}
\footnote{The strict inequality in the time definition is crucial. We cannot use $t \geq \tau_{\frac{1}{2}}$, otherwise one has $\tau_1 = \tau_{\frac{1}{2}}$.}
$$
where
\begin{equation}
\label{eq:dynavi}
\forall t > \tau_{\frac{1}{2}},
\: \forall \varphi \in \mathbb{R}, 
\: (b(x(t),t) - m \ddot x (t)) (\varphi - \dot x(t)) + f_d |\dot x(t)| \leq f_d |\varphi|
\end{equation}
with the initial condition
$$
x(\tau_{\frac{1}{2}}) = x_0, 
\: \quad 
\dot x(\tau_{\frac{1}{2}}) = 0. 
$$
This class of VI is standard and the theory tells that such a problem has one and only one solution in $C^1$ on the interval 
$[\tau _ {\frac {1} {2}},\infty)$. See for instance \cite{MR0348562} for a proof of this general result. Actually, we will see that 
$\tau_1 > \tau_{\frac{1}{2}}$. We will discuss the properties of the solution in the next section. If $\tau_1= \infty$ then $x$ remains in a dynamic state forever, otherwise the same procedure can be repeated with initial condition $x(\tau_1)$ that we denote $x_1$ at time $\tau_1$ where $| b(x_1,\tau_1)| \leq f_s$. By induction, we can define similarly the entrance times in static phase $\tau_j$ with corresponding states $x_j$ and the entrance times of dynamic phase $\tau_{j+\frac{1}{2}}$.
\subsection{Properties, proofs and explicit construction of a dynamic phase}
\begin{proposition}
\label{prop2-1}
An entrance time in dynamic phase $\tau_{j+\frac{1}{2}}$ is separated from its consecutive entrance time in static phase $\tau_{j+1}$, i.e. $\tau_{j+\frac{1}{2}} < \tau_{j+1}$. Moreover, on the subinterval $(\tau_{j+\frac{1}{2}},\tau_{j+1})$, we have $\textup{sign}(\dot x(t)) \neq 0$ a.e. and the trajectory
$x(t)$ is the solution of 
\begin{equation}
\label{eq:frictionpp}
\begin{cases}
m \ddot x (t) + \textup{sign}(\dot x(t))f_d = b(x(t),t) \: a.e. t >0\\
x(\tau_{j+\frac{1}{2}}) = x_j, \: \dot x(\tau_{j+\frac{1}{2}}) = 0.
\end{cases}
\end{equation}
\end{proposition}
\begin{proof}
Since the function $t \mapsto b(x(t),t)$ is continuous, for $\epsilon \triangleq \dfrac{f_s-f_d}{2}$ there is a $\delta_\epsilon>0$
such that 
$$
\forall t \in (\tau_{j+\frac{1}{2}}, \tau_{j+\frac{1}{2}} + \delta_\epsilon),
\: \left | |b(x(t),t)|-|b(x_j,\tau_{j+\frac{1}{2}})| \right | \leq \epsilon.
$$
Moreover, the condition 
$$
|b(x_j,\tau_{j+\frac{1}{2}})| = f_s
$$
implies
$$
\forall t \in (\tau_{j+\frac{1}{2}}, \tau_{j+\frac{1}{2}} + \delta_\epsilon), \: |b(x(t),t)| \geq f_s - \frac{f_s-f_d}{2} > f_d.
$$
Now, we claim that a static phase cannot occur in the time interval $(\tau_{j+\frac{1}{2}}, \tau_{j+\frac{1}{2}} + \delta_\epsilon)$.
We proceed by contradiction. Assume there is an subinterval $I \subseteq (\tau_{j+\frac{1}{2}}, \tau_{j+\frac{1}{2}} + \delta_\epsilon)$ of positive Lebesgue measure such that $\forall t \in I, \dot x (t) = 0, \: a.e.$.
Then $\forall t \in I, \: \ddot x (t) = 0, \: a.e.$ and from the V.I., we obtain
$$
\forall t \in I \: a.e. ,\: \forall \varphi \in \mathbb{R}, \: b(x(t),t) \varphi \leq f_d |\varphi|
$$
which implies
$$
|b(x(t),t)| \leq f_d.
$$
This is a contradiction since $f_d < f_s$. As a consequence $\tau_{j+\frac{1}{2}} < \tau_{j+1}$. Similarly, if at a time $t^\star \in (\tau_{j+\frac{1}{2}}, \tau_{j+1})$ one has $\dot x (t^\star) = 0$, then necessarily 
 $$
|b(x(t^\star),t^\star)| > f_s
$$
and there cannot be a small interval around $t^\star$ such that $\dot x(t) = 0$. Therefore, zeros of $\dot x(t)$, if any, are isolated points on the interval $(\tau_{j+\frac{1}{2}}, \tau_{j+1})$. This implies that the function $\textup{sign}(x(t))$ is defined a.e. on $(\tau_{j+\frac{1}{2}}, \tau_{j+1})$. So, outside of a set of measure $0$, $\dot x(t)$ is either positive or negative.
Suppose $\dot x(t)>0$ then \eqref{eq:dynavi} becomes
$$
 \forall \varphi \in \mathbb{R}, 
\: (b(x(t),t) - m \ddot x (t)) (\varphi - \dot x(t)) + f_d \dot x(t) \leq f_d |\varphi|.
$$
Thus we can state 
$$
 \forall \varphi \geq 0, 
\: (b(x(t),t) - m \ddot x (t) - f_d) (\varphi - \dot x(t)) \leq 0.
$$
Since $\dot x(t) >0$, necessarily
$$ 
m \ddot x (t)  + f_d = b(x(t),t)
$$
Similarly when $\dot x(t) <0$, we can check that
$$ 
m \ddot x (t)  - f_d = b(x(t),t)
$$
and thus \eqref{eq:frictionpp} holds true. The proof has been completed.
\end{proof}
\begin{remark}
Of course the V.I. of type \eqref{eq:dynavi} on 
$(\tau_{j+\frac{1}{2}}, \tau_{j+1})$ is equivalent to \eqref{eq:frictionpp}
together with the fact that the zeros of $\dot x(t)$ are isolated points. 
But one needs to define $\tau_{j+1}$ which depends on the equation.
We cannot just write equation \eqref{eq:frictionpp} after $\tau_{j+\frac{1}{2}}$
because we cannot define the function $\textup{sign}(\dot x(t))$ without prior knowledge that the zeros of $\dot x(t)$ are isolated.
We need to define a well posed problem, valid for any time posterior to $\tau_{j+\frac{1}{2}}$.
Equation \eqref{eq:frictionpp} cannot be used.
The V.I. is an elegant way to fix this difficulty. The only alternative is to construct the trajectory in the
dynamic phase step by step, as we are going to do below. The V.I. is a synthetic way to define the solution,
without a lengthy construction. As we shall see, when $f_s = f_d$, it will also incorporates the static phase, and
thus avoids any sequence of intervals. 
\end{remark}
\subsection{A step by step method for the dynamic phase as an alternative to the V.I.}
\label{sec:theory}
We now check that we can also define a sequence of sub phases of the dynamic phase, in which the solution satisfies a standard differential equation. This procedure is an alternative to the V.I. but is much less synthetic. 
The first dynamic sub-phase starts at 
$
\tau_{\frac{1}{2}} \triangleq \inf \{ t \geq \tau_0, \: |b(x(t),t)| > f_s \}
$
where we recall that $\tau_0 \triangleq 0$. 
Define 
$$
x_0^0 \triangleq x(0), \quad \tau_0^0 \triangleq \tau_{\frac{1}{2}},
$$ 
thus
$$
|b(x_0^0,\tau_0^0)| = f_s.
$$
We set $\epsilon_0^0 \triangleq \textup{sign} (b(x_0^0,\tau_0^0))$ and consider the differential equation
\begin{equation}
\label{eq:eqdifstand}
\begin{cases}
& m \ddot x (t) = b(x(t),t) - \epsilon_0^0 f_d, \: t > \tau_0^0,\\
& x(\tau_0^0) = x_0^0, \quad \dot x(\tau_0^0) = 0.
\end{cases}
\end{equation} 
Then, we define 
$$
\tau_0^1 \triangleq \inf \{ t > \tau_0^0, \: \dot x(t) = 0 \}
\quad
\mbox{and}
\quad
x_0^1 \triangleq x(\tau_0^1).
$$
\begin{itemize}
\item
If $|b(x_0^1,\tau_0^1)| \leq f_s$ then the dynamic phase ends here and a new static phase starts.
We set 
$$
\tau_1 \triangleq \tau_0^1
\quad
\mbox{and}
\quad
x_1 \triangleq x_0^1.
$$
\item
Otherwise $|b(x_0^1,\tau_0^1)| > f_s$ and we set 
$$
\epsilon_0^1 \triangleq \textup{sign} (b(x_0^1,\tau_0^1))
$$
and again consider the differential equation
\begin{equation}
\label{eq:eqdifstand0}
\begin{cases}
& m \ddot x (t) = b(x(t),t) - \epsilon_0^1 f_d, \: t > \tau_0^1,\\
& x(\tau_0^1) = x_0^1, \quad \dot x(\tau_0^1) = 0
\end{cases}
\end{equation} 
and introduce 
$$
\tau_0^2 \triangleq \inf \{ t > \tau_0^1, \: \dot x(t) = 0 \}.
$$
We can repeat the same procedure several times and thus define a sequence $\tau_0^k, x_0^k, \epsilon_0^k$ for $k \leq k(0)$ where 
$$
k(0) \triangleq \inf \{ k \geq 1, \: |b(x_0^k,\tau_0^k)| \leq f_s \} .
$$
If $k(0) = \infty$ then $x$ remains in a dynamic phase forever otherwise a new static phase starts at 
$$
\tau_1 \triangleq \tau_0^{k(0)}
\quad
\mbox{and}
\quad
x_1 \triangleq x_0^{k(0)}.
$$
The dynamic phase on the interval $(\tau_{\frac{1}{2}},\tau_1)$ is thus divided into dynamic sub-phases $(\tau_0^k,\tau_0^{k+1})$ with
$$
\tau_0 \leq \tau_{\frac{1}{2}} = \tau_0^0 < \tau_0^1 < \cdots < \tau_0^{k(0)} = \tau_1.
$$
\end{itemize}
Similarly, if for any $j \geq 1$ we know $x_j$ and $\tau_j$ with $|b(x_j,\tau_j)| \leq f_s$.
The first dynamic sub phase starts at 
$$
\tau_j^0 \triangleq \tau_{j+\frac{1}{2}}, \; x_j^0 \triangleq x_j , \: \mbox{where} \: |b(x_j^0,\tau_j^0)| = f_s.
$$
For each $k \geq 0$, to define the dynamic sub phase starting at $\tau_j^k$, we set
$$
\epsilon_j^k \triangleq \textup{sign}(b(x_j^k,\tau_j^k))
$$
and consider the differential equation
\begin{equation}
\label{eq:eqdifstand}
\begin{cases}
& m \ddot x (t) = b(x(t),t) - \epsilon_j^k f_d, \: t > \tau_j^k,\\
& x(\tau_j^k) = x_j^k, \quad \dot x(\tau_j^k) = 0.
\end{cases}
\end{equation}
Then, we uniquely define $\tau_j^{k+1}$ by
$$
\tau_j^{k+1} \triangleq \inf \{ t > \tau_j^k, \dot x(t) = 0 \}.
$$
This procedure defines the dynamic sub phase $(\tau_j^k,\tau_j^{k+1})$.
On this subinterval $\epsilon_j^k = \textup{sign}(\dot x(t))$, so \eqref{eq:eqdifstand} is also the equation
\begin{equation}
m \ddot x (t) = b(x(t),t) - \textup{sign}(\dot x(t)) f_d.
\end{equation}
We set $x_j^{k+1} \triangleq x(\tau_j^{k+1})$ and we proceed as follows:
\begin{itemize}
\item
if $| b(x_j^{k+1},\tau_j^{k+1}) | \leq f_s$
then we start a new static phase and we set 
$$
\tau_{j+1} \triangleq \tau_j^{k+1} 
\:
\mbox{and}
\: 
x_{j+1} \triangleq x_j^{k+1}.
$$
\item
on the other hand, if $| b(x_j^{k+1},\tau_j^{k+1}) | > f_s$
then we start a new dynamic sub phase on the interval $(\tau_j^{k+1},\tau_j^{k+2})$.
\end{itemize}
We can define 
$$
k(j) \triangleq \inf \{ k \geq 1, \: |b(x_j^k,\tau_j^k)| \leq f_s \},
$$
thus the dynamic phase on the interval $(\tau_{j+\frac{1}{2}},\tau_{j+1})$ is divided into dynamic sub-phases $(\tau_j^k,\tau_j^{k+1})$ with
$$
\tau_j \leq \tau_{j+\frac{1}{2}} = \tau_j^0 < \tau_j^1 < \cdots < \tau_j^{k(j)} = \tau_{j+1}.
$$
This procedure gives explicitly the solution of the V.I.. 
We will see in the next section that, when considering the case $b(x,t) = K ( \beta T(t) - x)$,
we have the additional advantage that the solution is explicit in a dynamic sub phase.
\subsection{Formula for $x(t)$ on a dynamic sub phase when $b(x,t) = K ( \beta T(t) - x)$}
When $b(x,t) = K ( \beta T(t) - x)$, the differential equation \eqref{eq:eqdifstand} has an explicit solution.
For $t \in [ \tau_j^k, \tau_j^{k+1})$,
\begin{equation}
\label{eq:xexplicit}
x(t) 
= 
x_j^k \cos \omega \left (t - \tau_j^k \right )
+
\int_{\tau_j^k}^t \sin \omega (t - s) \dfrac{K \beta T(s) - \epsilon_j^k f_d}{m \omega} \textup{d} s
\end{equation}
and $\tau_j^{k+1}$ is defined by the equation
\begin{equation}
\label{eq:taujkp1}
x_j^k \sin \omega \left (\tau_j^{k+1} - \tau_j^k \right )
= 
\int_{\tau_j^k}^{\tau_j^{k+1}} \cos \omega (\tau_j^{k+1} - s) \dfrac{K \beta T(s) - \epsilon_j^k f_d}{m \omega} \textup{d} s.
\end{equation}
We then set 
\begin{equation}
\label{eq:xjkp1}
x_j^{k+1} \triangleq x(\tau_j^{k+1}) 
= x_j^k \cos \omega \left (\tau_j^{k+1} - \tau_j^k \right )
+
\int_{\tau_j^k}^{\tau_j^{k+1}} \sin \omega (\tau_j^{k+1} - s) \dfrac{K \beta T(s) - \epsilon_j^k f_d}{m \omega} \textup{d} s.
\end{equation}
\subsection{Algorithm}
The above discussion allows to define an algorithm to construct the trajectory $x(t)$ of the initial problem
\eqref{eq:friction1} and \eqref{eq:friction2} with $b(x,t)$ given in \eqref{eq:engapp}.
We construct the sequence $\{ x_j,\tau_j, \: j \geq 0 \}$ where
$$
|b(x_j, \tau_j)| \leq f_s.
$$
When $j = 0$, it corresponds to the initial condition $x_0 \triangleq x(0)$ and $\tau_0 \triangleq 0$.
For each $j \geq 0$, 
we first compute 
$$
\tau_{j+\frac{1}{2}} \triangleq \inf \{ t \geq \tau_j, \: |b(x_j,t)| > f_s \}
$$
and then we define a subsequence with respect to $k$.
\begin{itemize}
\item
Define 
$$
\tau_j^0 \triangleq \tau_{j+\frac{1}{2}}, 
\quad x_j^0 \triangleq x_j,
\quad \epsilon_j^0 \triangleq \textup{sign} (b(x_j^0,\tau_j^0)).
$$
\item
For each $k \geq 0$, using $\epsilon_j^k$, 
define $\tau_j^{k+1}$ using \eqref{eq:taujkp1} and $x_j^{k+1}$ using \eqref{eq:xjkp1}.
\item
If 
$$
|b(x_j^{k+1},\tau_j^{k+1})| \leq f_s
$$
then
$$
\tau_{j+1} \triangleq \tau_j^k \: \mbox{and} \: x_{j+1} \triangleq x_j^{k+1}
$$
otherwise we define 
$$
\epsilon_j^{k+1} \triangleq \textup{sign} (b(x_j^{k+1},\tau_j^{k+1}))
$$
and repeat the procedure with the definition of $\tau_j^{k+2},x_j^{k+2}$.
\end{itemize}
The sub interval $(\tau_j, \tau_{j+\frac{1}{2}})$ is a static or stick subphase, and the sub interval $(\tau_{j+\frac{1}{2}}, \tau_{j+1})$ 
is a dynamic or slip phase, itself subdivided into dynamic subphases. The trajectory $x(t)$ is completely defined by this cascade of phases and subphases.

\subsection{Quasistatic approximation}
Since Equation \eqref{eq:taujkp1} is transcendental, 
we need also an algorithm to solve it.
We are going to state an approximation which simplifies considerably the calculations.
This approximation is called quasistatic which means that the excitation variation is slow enough to be neglected during any dynamic phase.
From then on, an interesting consequence of this approximation is that there are no subphases in the dynamic phases.
We will see below
that, with the definition of sub phases of the Section 2.4 in mind, one has  $x_j^1$ and $\tau_j^1$ satisfy 
$$
|K (\beta T(\tau_j^1) - x_j^1)| \leq f_s.
$$
So there is only one sequence $\tau_j,x_j$ and once $\tau_{j+\frac{1}{2}}$ is defined, 
the equation for $\tau_j^1$ becomes simply an equation for $\tau_{j+1}$ 
which is
\begin{equation}
\label{eq:taujp1}
x_j \sin \omega \left (\tau_{j+1} - \tau_j \right )
= 
\int_{\tau_j}^{\tau_{j+1}} \cos \omega (\tau_{j+1} - s) \dfrac{K \beta T(s) - \epsilon_j f_d}{m \omega} \textup{d} s.
\end{equation}
with 
$$
\epsilon_j \triangleq \textup{sign} (\beta T(\tau_{j+\frac{1}{2}}) - x_j).
$$
Next finding $x_{j+1}$ uses Equation \eqref{eq:xjkp1} as follows:
\begin{equation}
\label{eq:xjp1}
x_{j+1} \triangleq x_j \cos \omega \left (\tau_{j+1} - \tau_{j+\frac{1}{2}} \right )
+
\int_{\tau_{j+\frac{1}{2}}}^{\tau_{j+1}} \sin \omega (\tau_{j+1} - s) \dfrac{K \beta T(s) - \epsilon_j f_d}{m \omega} \textup{d} s.
\end{equation}
The quasistatic approximation consists in making the approximation
$$
\forall s \in (\tau_{j+\frac{1}{2}},\tau_{j+1}), \quad T(s) = T(\tau_{j+\frac{1}{2}}).
$$
We will use the notation $T_{j+\frac{1}{2}} \triangleq T(\tau_{j+\frac{1}{2}})$.
From \eqref{eq:xjp1}, we obtain
$$
\sin (\omega (\tau_{j+1}-\tau_{j+\frac{1}{2}})) = 0
$$
which implies 
$$
\tau_{j+1} = \tau_{j+\frac{1}{2}} + \frac{\pi}{\omega}.
$$
We turn to \eqref{eq:xjp1} which simplifies considerably using 
\begin{equation}
\label{eq:interpret}
|\beta T_{j+\frac{1}{2}} - x_j| = \frac{f_s}{K},
\end{equation}
and we obtain
\begin{equation}
\label{eq:2-10}
x_{j+1} = x_j + 2 \epsilon_j \frac{f_s-f_d}{K}.
\end{equation}
Of course, we need to find 
$$
\tau_{j+\frac{1}{2}} = \inf \{ t > \tau_j, \: |\beta T(t) - x_j| > \frac{f_s}{K}\}
$$
and then define 
$$
\tau_{j+1} = \tau_{j+\frac{1}{2}} + \frac{\pi}{\omega}.
$$
Also
$$
\epsilon_{j+1} \triangleq 
\textup{sign}(\beta T_{j+\frac{3}{2}} - x_{j+1}).
$$
When we are in a static phase at the value $x_j$, starting at $\tau_j$,
we can interpret the condition \eqref{eq:interpret} as a condition that the future temperature 
must satisfy to initiate a new dynamic phase.

\subsection{Case $f \triangleq f_s=f_d$ }
When $f \triangleq f_s=f_d$, we do not need to define the static phases as in Section 2.2 above.
Indeed, we have seen in the proof of Proposition \ref{prop2-1}, that if $\dot x(t)=0$ on a set of positive measure, then necessarily we have $|\beta T(t)-x(t)|\leq\dfrac{f}{K}$.
This means that during a static phase, the V.I. is also satisfied.
Therefore, the problem is simply 
\begin{equation}
\label{eq:2-17}
\forall\varphi\in \mathbb{R}, \;
	(m \ddot x(t)-K(\beta T(t)-x(t)))(\dot x(t)-\varphi)+f |\dot x(t)|\leq f |\varphi|
\end{equation}
with
\[
	x(0)=x_{0}, \:\dot x(0)=0.
\]
In the case $f_s=f_d,$ we also see from the approximation formula \eqref{eq:2-10} that $x_{j}$ is constant.
So the system enters in the static mode at the same point.
Since in our assumption $\dot x(0)=0$ and $x(0)=x_{0},$ we have $x_{j}=x_{0}$.
This is only an approximation.
The alternance of static and dynamic modes follows the sequence of times, $\tau_{j},$$\tau_{j+\frac{1}{2}}$, such that 
\begin{equation}
	\tau_{j+\frac{1}{2}}=\inf \left \{t>\tau_{j}\:|\;T(t) \notin 
	\left [ \frac{x_0}{\beta} - \frac{f}{\beta K},\frac{x_0}{\beta} + \frac{f}{\beta K} \right ] \right \}
	\label{eq:2-14}
\end{equation}
\[
	\tau_{j+1}-\tau_{j+\frac{1}{2}}=\dfrac{\pi}{\omega}
\]
with $\tau_{0}=0$.
What is not an approximation is the following discussion.
In a static mode $\tau_{j}\leq t\leq$$\tau_{j+\frac{1}{2}},$ we have 
\begin{equation}
	|\beta T(t)-x(t)|\leq\dfrac{f}{K}
\end{equation}
and in a dynamic mode,$\tau_{j+\frac{1}{2}}\leq t\leq\tau_{j+1}$ we have the equation 
\begin{equation}
	m \ddot x(t)=K(\beta T(t)-x(t))-\epsilon_{j}f_s\label{eq:2-16}
\end{equation}
and $x(t)$ is $C^{1}$, with $\ddot x(t)$ locally bounded. 

\section{Simulations}
\subsection{Step by step Euler method : a discrete version of the two-phase model}
\label{numerics}
Let $T>0, N \in \mathbb{N}^\star$ and $\{ t_n \}_{n=0}^N$ a family of time which discretizes $[0,T]$
such that $t_n \triangleq n h$ where $h \triangleq \frac{T}{N}$. We will construct 
a sequence $\{ (X_{t_n}^h, \dot X_{t_n}^h) \}_{n=0}^N$
where for each $n$, $(X_{t_n}^h, \dot X_{t_n}^h)$ is meant to approximate $(x(t_n),\dot x(t_n))$.
\subsubsection{Case $f_d = f_s$}
When $f \triangleq f_d = f_s$, the dynamics of $x(t)$ is governed by the variational inequality \eqref{eq:vi}.
Thus a finite difference method leads to a discrete V.I. as follows:
$$
X_{t_0}^h \triangleq x_0, 
\: \dot X_{t_0}^h  \triangleq 0
$$
then for $n = 0, 1, \cdots , N-1$
$$
X_{t_{n+1}}^h \triangleq  X_{t_n}^h + h \dot X_{t_n}^h
$$
and $\dot X_{t_{n+1}}^h$ is defined as the unique solution of the discrete V.I.
$$
\forall \varphi \in \mathbb{R}, \: 
\left ( b(X_{t_n}^h,t_n) - m \left ( \dfrac{\dot X_{t_{n+1}}^h - \dot X_{t_n}^h}{h} \right ) \right ) (\varphi - \dot X_{t_{n+1}}^h )
+ f | \dot X_{t_{n+1}}^h| \leq f |\varphi|
$$
which is equivalent to defining $\dot X_{t_{n+1}}^h$ using a discrete version of a two-phase model in this way:
\begin{itemize}
\item
if
$$
\left | \dot X_{t_n}^h + \dfrac{h}{m} b(X_{t_n}^h,t_n) \right | \leq \frac{h}{m} f 
$$
then
$$
\dot X_{t_{n+1}}^h \triangleq 0.
$$
\item
if 
$$
\left | \dot X_{t_n}^h + \dfrac{h}{m} b(X_{t_n}^h,t_n) \right | > \frac{h}{m} f 
$$
then
$$
\dot X_{t_{n+1}}^h \triangleq \dot X_{t_n}^h + \dfrac{h}{m} \left ( b(X_{t_n}^h,t_n) - f \epsilon_{t_n}^h \right ), 
\: \mbox{ where } \: \epsilon_{t_n}^h \triangleq \textup{sign}\left (\dot X_{t_n}^h + \dfrac{h}{m} b(X_{t_n}^h,t_n) \right ).
$$
\end{itemize}
\subsubsection{\chm{Case} $f_d < f_s$}
When $f_d$ and $f_s$ are not identical, the dynamics of $x(t)$ is governed by the two-phase model for which the dynamic phase remains governed by \eqref{eq:vi} but not the static phase. The finite difference method becomes

$$
X_{t_0}^h \triangleq x_0, 
\: \dot X_{t_0}^h  \triangleq 0
$$
then for $n = 0, 1, \cdots , N-1$
$$
X_{t_{n+1}}^h \triangleq  X_{t_n}^h + h \dot X_{t_n}^h
$$
and $\dot X_{t_{n+1}}^h$ is defined using a discrete version of a two-phase model:
\begin{itemize}
\item
if
$$
\left | \dot X_{t_n}^h + \dfrac{h}{m} b(X_{t_n}^h,t_n) \right | \leq \frac{h}{m} f_s 
$$
then
$$
\dot X_{t_{n+1}}^h \triangleq 0.
$$
\item
if 
$$
\left | \dot X_{t_n}^h + \dfrac{h}{m} b(X_{t_n}^h,t_n) \right | > \frac{h}{m} f_s 
$$
then $\dot X_{t_{n+1}}^h$ is defined as the unique solution of the discrete V.I.
$$
\forall \varphi \in \mathbb{R}, \: 
\left ( b(X_{t_n}^h,t_n) - m \left ( \dfrac{\dot X_{t_{n+1}}^h - \dot X_{t_n}^h}{h} \right ) \right ) (\varphi - \dot X_{t_{n+1}}^h )
+ f_d | \dot X_{t_{n+1}}^h| \leq f_d |\varphi|.
$$
which means
$$
\dot X_{t_{n+1}}^h \triangleq \dot X_{t_n}^h + \dfrac{h}{m} \left ( b(X_{t_n}^h,t_n) - f_d \epsilon_{t_n}^h \right ), 
\: \epsilon_{t_n}^h \triangleq \textup{sign}\left (\dot X_{t_n}^h + \dfrac{h}{m} b(X_{t_n}^h,t_n) \right ).
$$
\end{itemize}
\subsubsection{Simulation of the response of a system with dry friction to harmonic excitation}
We consider the model studied by Shaw in \cite{MR0852197}.
We apply the algorithm above to simulate $x(t)$ satisfying \eqref{eq:friction1}
where the right hand side is of the form
\begin{equation}
\label{eq:shaw}
b(x(t),\dot x(t),t) \triangleq \beta \cos (\omega t) - 2 \alpha \dot x(t) - x(t).
\end{equation}
To ensure sticking motions, we chose the parameters as follows
\begin{equation}
\label{parameters}
\tag{$\mathcal{P}$}
m = 1,
\: f_d = 1, 
\: f_s = 1.2, 
\: \alpha = 0,
 \: \beta = 6,
 \: \omega = \frac{1}{2}. 
 \end{equation}
This choice is inspired from Figure 8 p. 315 in \cite{MR0852197}. The numerical results are shown in Figure \ref{fig:shaw}.
We recover the results obtained by \cite{MR0852197}.

\begin{figure}[h!]
	\centering
\begin{tikzpicture}[scale=1]
\begin{axis}[legend style={at={(0,1)},anchor=north west}, compat=1.3,
  xmin=-1, xmax=28,ymin=-7.5,ymax=7.5,
  xlabel= {Time $t$},
  ylabel= {Displacement $x(t)$}]
  
\addplot[thick,dashdotted,color=black,mark=none,mark size=1.5pt] table [x index=0, y index=1]{CASE-RHO-0/case1_s1.txt}; 
\addplot[thick,dashdotted,color=black,mark=none,mark size=1.5pt] table [x index=0, y index=1]{CASE-RHO-0/case1_s2.txt}; 
\addplot[thick,solid,color=black,mark=none,mark size=1.5pt] table [x index=0, y index=1]{CASE-RHO-0/case1_d1.txt}; 
\addplot[thick,solid,color=black,mark=none,mark size=1.5pt] table [x index=0, y index=1]{CASE-RHO-0/case1_d2.txt}; 

\addplot[color=gray,solid] coordinates {(0,0)(26,0)};	

\addplot[thick,color=gray,dotted] coordinates {(0,0)(0,6)};
\addplot[thick,color=gray,dotted] coordinates {(2.57,0)(2.57,6)};

\addplot[thick,color=gray,dotted] coordinates {(13.38,0)(13.38,-6.24223)};
\addplot[thick,color=gray,dotted] coordinates {(14.86,0)(14.86,-6.24223)};

\addplot[thick,color=gray,dotted] coordinates {(25.74,0)(25.74,6.22228)};

\draw[fill = black] (axis cs: 0,0) circle (1.5pt);
\node at (axis cs: 0,-1) {$\tau_0$};

\draw[fill = black] (axis cs: 2.58,0) circle (1.5pt);
\node at (axis cs: 2.58,-1) {$\tau_{\frac{1}{2}}$};

\draw[fill = black] (axis cs: 13.38,0) circle (1.5pt);
\node at (axis cs: 13.30,1) {$\tau_1$};

\draw[fill = black] (axis cs: 14.86,0) circle (1.5pt);
\node at (axis cs: 14.86,1) {$\tau_{\frac{3}{2}}$};

\draw[fill = black] (axis cs: 25.74,0) circle (1.5pt);
\node at (axis cs: 25.74,-1) {$\tau_2$};

\end{axis}
\end{tikzpicture}
\begin{tikzpicture}[scale=1]
\begin{axis}[legend style={at={(0,1)},anchor=north west}, compat=1.3,
  xmin=-1, xmax=28,ymin=-3,ymax=3,
  xlabel= {Time $t$},
  ylabel= {Velocity $\dot x(t)$}]

\addplot[thick,dashdotted,color=black,mark=none,mark size=1.5pt] table [x index=0, y index=2]{CASE-RHO-0/case1_s1.txt}; 
\addplot[thick,dashdotted,color=black,mark=none,mark size=1.5pt] table [x index=0, y index=2]{CASE-RHO-0/case1_s2.txt}; 
\addplot[thick,solid,color=black,mark=none,mark size=1.5pt] table [x index=0, y index=2]{CASE-RHO-0/case1_d1.txt}; 
\addplot[thick,solid,color=black,mark=none,mark size=1.5pt] table [x index=0, y index=2]{CASE-RHO-0/case1_d2.txt};

\addplot[color=gray,solid] coordinates {(0,0)(26,0)};	

\draw[fill = black] (axis cs: 0,0) circle (1.5pt);
\node at (axis cs: 0,-0.5) {$\tau_0$};

\draw[fill = black] (axis cs: 2.58,0) circle (1.5pt);
\node at (axis cs: 2.58,-0.5) {$\tau_{\frac{1}{2}}$};

\draw[fill = black] (axis cs: 13.38,0) circle (1.5pt);
\node at (axis cs: 13.30,0.5) {$\tau_1$};

\draw[fill = black] (axis cs: 14.86,0) circle (1.5pt);
\node at (axis cs: 14.86,0.5) {$\tau_{\frac{3}{2}}$};

\draw[fill = black] (axis cs: 25.74,0) circle (1.5pt);
\node at (axis cs: 25.74,-0.5) {$\tau_2$};

\end{axis}
\end{tikzpicture}
\begin{tikzpicture}[scale=1]
\begin{axis}[legend style={at={(0,1)},anchor=north west}, compat=1.3,
  xmin=-3, xmax=3,ymin=-2.5,ymax=2.5,
  xlabel= {Velocity $\dot x(t)$},
  ylabel= {Force $b(x(t),\dot x(t),t)$}]
\addplot[thick,dashdotted,color=black,mark=none,mark size=3pt] table [x index=2, y index=4]{CASE-RHO-0/case1_s1.txt};
\draw[fill = black] (axis cs: 0,0) circle (1.5pt);
\node at (axis cs: 0.3,0) {$\tau_0$};
\draw[fill = black] (axis cs: 0,-1.2) circle (1.5pt);
\node at (axis cs: 0.3,-1.2) {$\tau_{\frac{1}{2}}$};
\draw[fill = black] (axis cs: 0,0.365928) circle (1.5pt);
\node at (axis cs: 0.3,0.365928) {$\tau_1$};
\addplot[thick,color=black,mark=none,mark size=3pt] table [x index=2, y index=4]{CASE-RHO-0/case1_d1.txt}; 
\addplot[thick,dashdotted,color=black,mark=none,mark size=3pt] table [x index=2, y index=4]{CASE-RHO-0/case1_s2.txt};
\draw[fill = black] (axis cs: 0,1.2) circle (1.5pt);
\node at (axis cs: -0.3,1.2) {$\tau_{\frac{3}{2}}$};
\addplot[thick,color=black,mark=none,mark size=3pt] table [x index=2, y index=4]{CASE-RHO-0/case1_d2.txt}; 
\draw[fill = black] (axis cs: 0,-0.291287) circle (1.5pt);
\node at (axis cs: -0.25,-0.291287) {$\tau_2$};

\end{axis}
\end{tikzpicture}
\caption{Simulation of trajectories.
Typical trajectory of a solution $x(t)$ with an harmonic forcing (Equation \eqref{eq:friction1} with \eqref{eq:shaw}).
The numerical results have been obtained using the numerical scheme shown in Section \ref{numerics}.
The time intervals enclosed by
$[\tau_0,\tau_{\frac{1}{2}})$ 
and   
$[\tau_1,\tau_{\frac{3}{2}})$ 
correspond to static phases whereas 
the intervals 
$[\tau_{\frac{1}{2}},\tau_1)$ 
and   
$[\tau_{\frac{3}{2}},\tau_2)$
correspond to dynamic phases. 
}
\label{fig:shaw}
\end{figure}
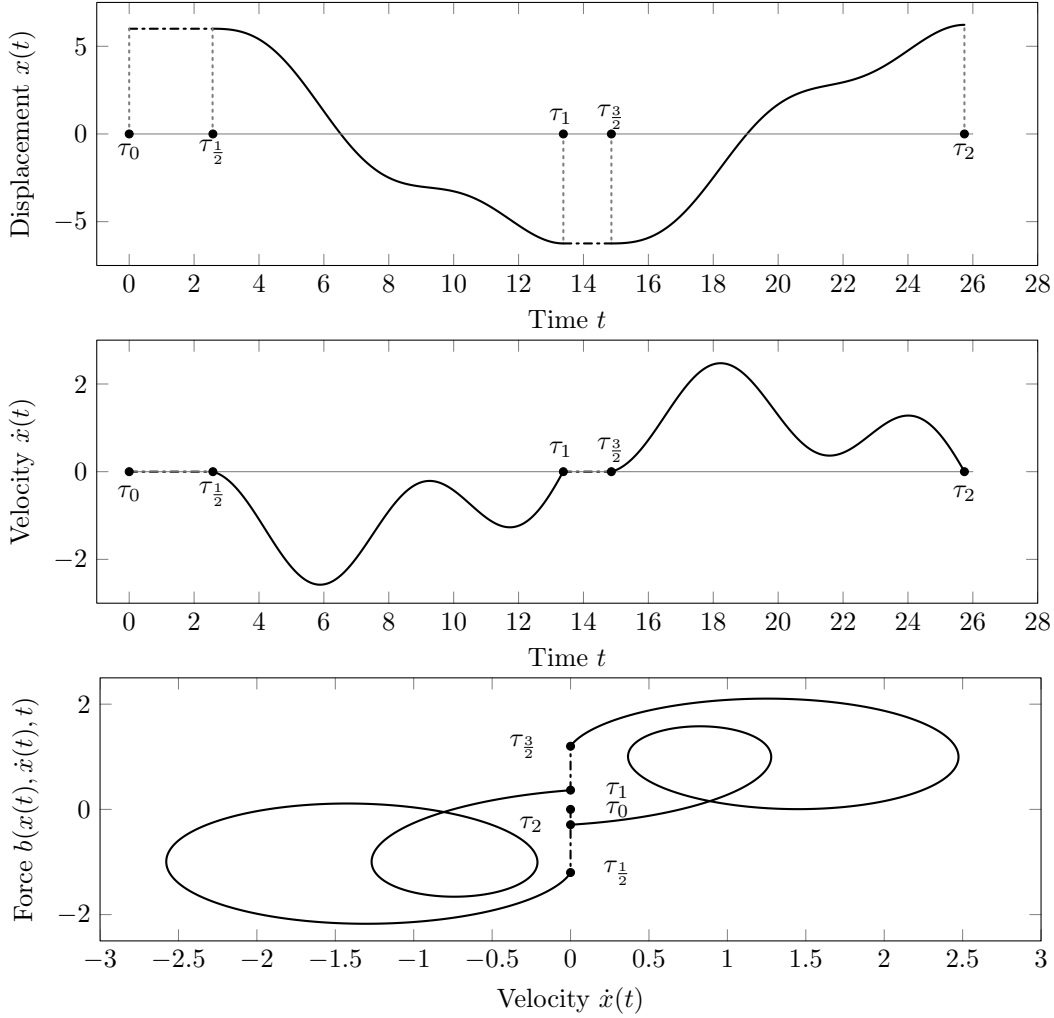

\subsubsection{Simulation of the response of a system with dry friction to a random perturbation}
Using the numerical scheme above, the behaviour of \eqref{eq:friction1} - \eqref{eq:friction2}
has been simulated with a temperature function chosen as 
$T(t) = \cos(\omega t) + \rho v(t)$ where $v(t)$ is an Ornstein Uhlenbeck noise in the sense that 
$$
\dot v = - v + \dot w, \: \mbox{ where } \: w \: \mbox{ is a Wiener process.} 
$$
This can be seen as a random perturbation of the model on the section above.
Here $\rho$ is a small parameter. Figure \ref{fig:temporal_and_phase_plan} display the result of the simulation.
we used the parameters \eqref{parameters} and $K = 1, \rho = 0.25$.
Figure \ref{fig:temporal_and_phase_plan} displays 
the displacement $x(t)$ in response to $\beta T(t)$,
and the velocity $\dot x (t)$ together with the forces $\mathbb{F}(\dot x(t))$ and $K(\beta T(t)-x(t))$.
Figure \ref{fig:temporal_and_phase_plan} \textbf{C} displays the evolution of $(\beta T(t)-x(t), \dot x(t))$ with respect to $t$ in a sort of phase plan.
The static phases occur during the time intervals $(\tau_j,\tau_{j+\frac{1}{2}})$
and are noticeable in Figure \ref{fig:temporal_and_phase_plan} when $x(t)$ is constant (Figure \ref{fig:temporal_and_phase_plan} \textbf{A}) 
or $\dot x(t) = 0$ (Figure \ref{fig:temporal_and_phase_plan} \textbf{B}) over time intervals with positive Lebesgue measure.
In contrast, the dynamic phases occur during the time intervals $(\tau_{j+\frac{1}{2}},\tau_{j+1})$
and are noticeable in Figure \ref{fig:temporal_and_phase_plan} \textbf{B} as the peaks and in Figure \ref{fig:temporal_and_phase_plan} \textbf{C} as the cycles. In this simulation, there are subphases in the dynamic phases.
The likeliness of occurrence of dynamic sub phases is highly dependent on important variations of the temperature $T(t)$.
 
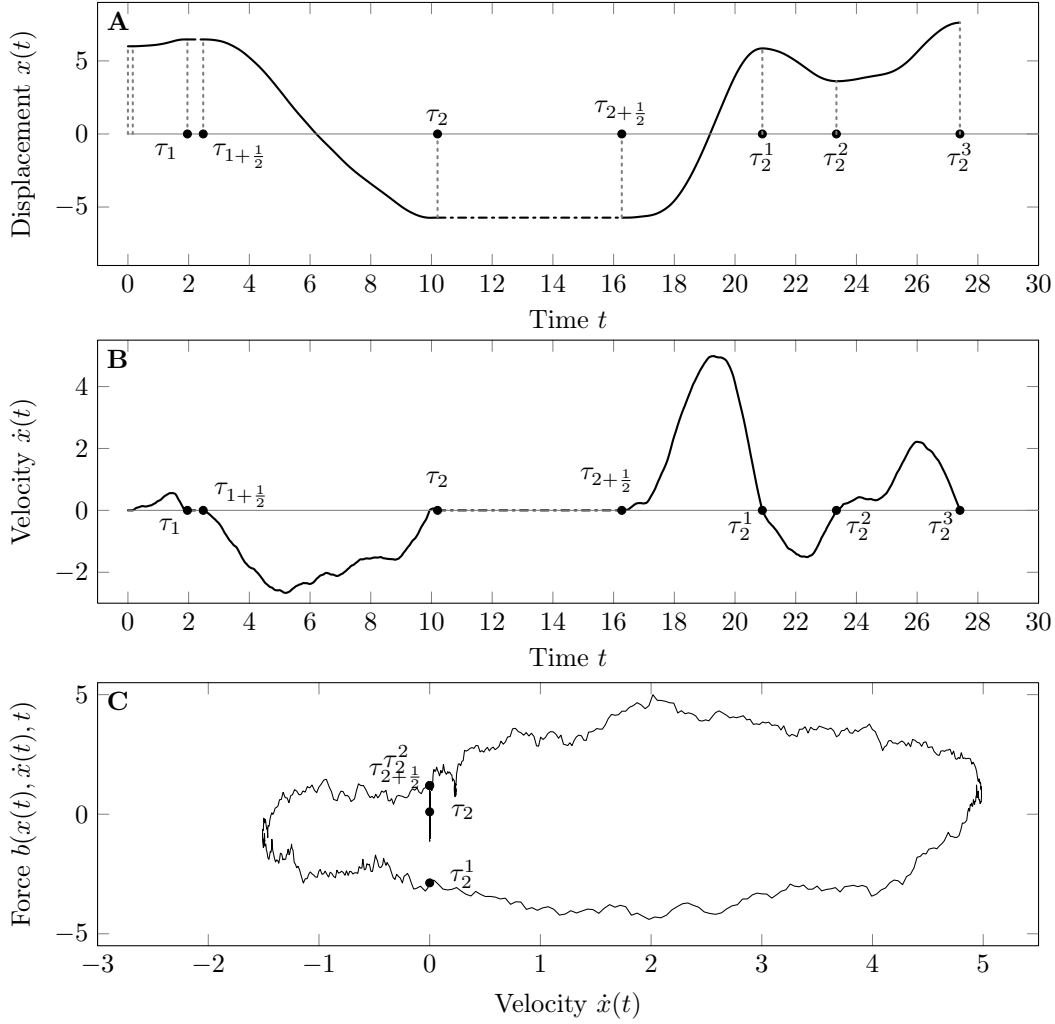
\begin{figure}[h!]
\begin{tikzpicture}[scale=1]
\begin{axis}[legend style={at={(0,1)},anchor=north west}, compat=1.3,
  xmin=-1, xmax=30,ymin=-9,ymax=9,
  xlabel= {Time $t$},
  ylabel= {Displacement $x(t)$}]
 \node[below right] at (current axis.north west) {\chm{\textbf{A}}};
\addplot[thick,dashdotted,color=black,mark=none,mark size=1.5pt] table [x index=0, y index=1]{CASE-RHO-025/case1rho025s1.txt}; 
\addplot[thick,dashdotted,color=black,mark=none,mark size=1.5pt] table [x index=0, y index=1]{CASE-RHO-025/case1rho025s2.txt}; 
\addplot[thick,dashdotted,color=black,mark=none,mark size=1.5pt] table [x index=0, y index=1]{CASE-RHO-025/case1rho025s3.txt}; 

\addplot[thick,color=black,mark=none,mark size=1.5pt] table [x index=0, y index=1]{CASE-RHO-025/case1rho025d1.txt}; 
\addplot[thick,color=black,mark=none,mark size=1.5pt] table [x index=0, y index=1]{CASE-RHO-025/case1rho025d2.txt}; 
\addplot[thick,color=black,mark=none,mark size=1.5pt] table [x index=0, y index=1]{CASE-RHO-025/case1rho025d3.txt};

\addplot[color=gray,solid] coordinates {(0,0)(50,0)};	

\addplot[thick,color=gray,dotted] coordinates {(0,0)(0,6)};
\addplot[thick,color=gray,dotted] coordinates {(0.16,0)(0.16,6)};
%
\addplot[thick,color=gray,dotted] coordinates {(1.96,0)(1.96,6.46726)};
\addplot[thick,color=gray,dotted] coordinates {(2.48,0)(2.48,6.46726)};
%
\addplot[thick,color=gray,dotted] coordinates {(10.20,0)(10.20,-5.7267)};
\addplot[thick,color=gray,dotted] coordinates {(16.27,0)(16.27,-5.7267)};
%
%
%
%
%
\draw[fill = black] (axis cs: 1.96,0) circle (1.5pt) node[below left] {$\tau_1$};;

\draw[fill = black] (axis cs: 2.48,0) circle (1.5pt) node[below right] {\chm{$\tau_{1+\frac{1}{2}}$}};

\draw[fill = black] (axis cs: 10.20,0) circle (1.5pt) node[above] {$\tau_2$};

\draw[fill = black] (axis cs: 16.27,0) circle (1.5pt) node[above] {$\tau_{2+\frac{1}{2}}$};

\draw[fill = black] (axis cs: 20.9,0) circle (1.5pt) node[below] {$\tau_2^1$};

\addplot[thick,color=gray,dotted] coordinates {(20.9,0)(20.9,5.855)};

\draw[fill = black] (axis cs: 23.34,0) circle (1.5pt) node[below] {$\tau_2^2$};
\addplot[thick,color=gray,dotted] coordinates {(23.34,0)(23.34,3.611)};

\draw[fill = black] (axis cs: 27.41,0) circle (1.5pt) node[below] {$\tau_2^3$};
\addplot[thick,color=gray,dotted] coordinates {(27.41,0)(27.41,7.615)};

\end{axis}
\end{tikzpicture}
\begin{tikzpicture}[scale=1]
\begin{axis}[legend style={at={(0,1)},anchor=north west}, compat=1.3,
  xmin=-1, xmax=30,ymin=-3,ymax=5.5,
  xlabel= {Time $t$},
  ylabel= {Velocity $\dot x(t)$}]
 \node[below right] at (current axis.north west)  {\chm{\textbf{B}}};
\addplot[thick,dashdotted,color=black,mark=none,mark size=1.5pt] table [x index=0, y index=2]{CASE-RHO-025/case1rho025s1.txt}; 
\addplot[thick,dashdotted,color=black,mark=none,mark size=1.5pt] table [x index=0, y index=2]{CASE-RHO-025/case1rho025s2.txt}; 
\addplot[thick,dashdotted,color=black,mark=none,mark size=1.5pt] table [x index=0, y index=2]{CASE-RHO-025/case1rho025s3.txt}; 

\addplot[thick,color=black,mark=none,mark size=1.5pt] table [x index=0, y index=2]{CASE-RHO-025/case1rho025d1.txt}; 
\addplot[thick,color=black,mark=none,mark size=1.5pt] table [x index=0, y index=2]{CASE-RHO-025/case1rho025d2.txt}; 
\addplot[thick,color=black,mark=none,mark size=1.5pt] table [x index=0, y index=2]{CASE-RHO-025/case1rho025d3.txt}; 

\addplot[color=gray,solid] coordinates {(0,0)(50,0)};	

\draw[fill = black] (axis cs: 1.96,0) circle (1.5pt);
\node at (axis cs: 1.4,-0.5) {$\tau_1$};

\draw[fill = black] (axis cs: 2.48,0) circle (1.5pt);
\node at (axis cs: 3.75,0.5) {$\tau_{1+\frac{1}{2}}$};

\draw[fill = black] (axis cs: 10.20,0) circle (1.5pt);
\node at (axis cs: 10.20,1) {$\tau_2$};

\draw[fill = black] (axis cs: 16.27,0) circle (1.5pt);
\node at (axis cs: 15.8,1) {$\tau_{2+\frac{1}{2}}$};

\draw[fill = black] (axis cs: 20.9,0) circle (1.5pt);
\node at (axis cs: 20.2,-0.5) {$\tau_2^1$};

\draw[fill = black] (axis cs: 23.34,0) circle (1.5pt);
\node at (axis cs: 24.1,-0.5) {$\tau_2^2$};

\draw[fill = black] (axis cs: 27.41,0) circle (1.5pt);
\node at (axis cs: 26.8,-0.5) {$\tau_2^3$};

%
%
%
%

\end{axis}
\end{tikzpicture}
\begin{tikzpicture}[scale=1]
\begin{axis}[legend style={at={(0,1)},anchor=north west}, compat=1.3,
  xmin=-3, xmax=5.5,ymin=-5.5,ymax=5.5,
  xlabel= {Velocity $\dot x(t)$},
  ylabel= {Force $b(x(t),\dot x(t),t)$}]
 \node[below right] at (current axis.north west)  {\chm{\textbf{C}}};
\addplot[dashdotted,color=black,mark=none,mark size=1.5pt] table [x index=2, y index=4]{CASE-RHO-025/case1rho025s3.txt}; 

\addplot[color=black,mark=none,mark size=1.5pt] table [x index=2, y index=4]{CASE-RHO-025/case1rho025d3-phaseplan.txt}; 

\draw[fill = black] (axis cs: 0,0.0999101) circle (1.5pt);
\node at (axis cs: 0.3,0.0999101) {$\tau_2$};

\draw[fill = black] (axis cs: 0,1.2) circle (1.5pt);
\node at (axis cs: -0.3,1.6) {$\tau_{2+\frac{1}{2}}$};

\draw[fill = black] (axis cs: 0,1.2) circle (1.5pt);
\node at (axis cs: -0.3,2.3) {$\tau_2^2$};

\draw[fill = black] (axis cs: 0,-2.87106) circle (1.5pt);
\node at (axis cs: 0.3,-2.5) {$\tau_2^1$};
  


\end{axis}
\end{tikzpicture}

\caption{Simulation of trajectories. Typical trajectory of a solution $x(t)$ where the temperature function $T$ is given by $T(t) = \cos(\omega t) + \rho v(t)$ where $v$ is an Ornstein Uhlenbeck process, $\rho = 0.25$. The $\tau_\bullet$ and $\tau_\bullet^\bullet$ pinned on the trajectory refer to the times defined in section \ref{sec:theory}. 
		}
	\label{fig:temporal_and_phase_plan}
\end{figure}

\newpage

\section{Experimental campaign}
\label{sec:experimental}
In general, mechanical properties of real world structures are not known but in some cases it is possible to infer them from observations (experimental data). Here, we focus on the behavior of such a structure (bridge component) subjected to changes of temperatures over time, see Figure~\ref{fig:bridge}. Engineers are interested in inferring the initial condition of displacement and other properties of the structure such as the stiffness, the magnitude of the static and dynamic friction forces and the dilatation with respect to the temperature.
In terms of our model, it corresponds to adjusting the parameters $z_0,K,f_s,f_d$ and $\beta$ to fit a set of observations.  
\begin{figure}[h!]
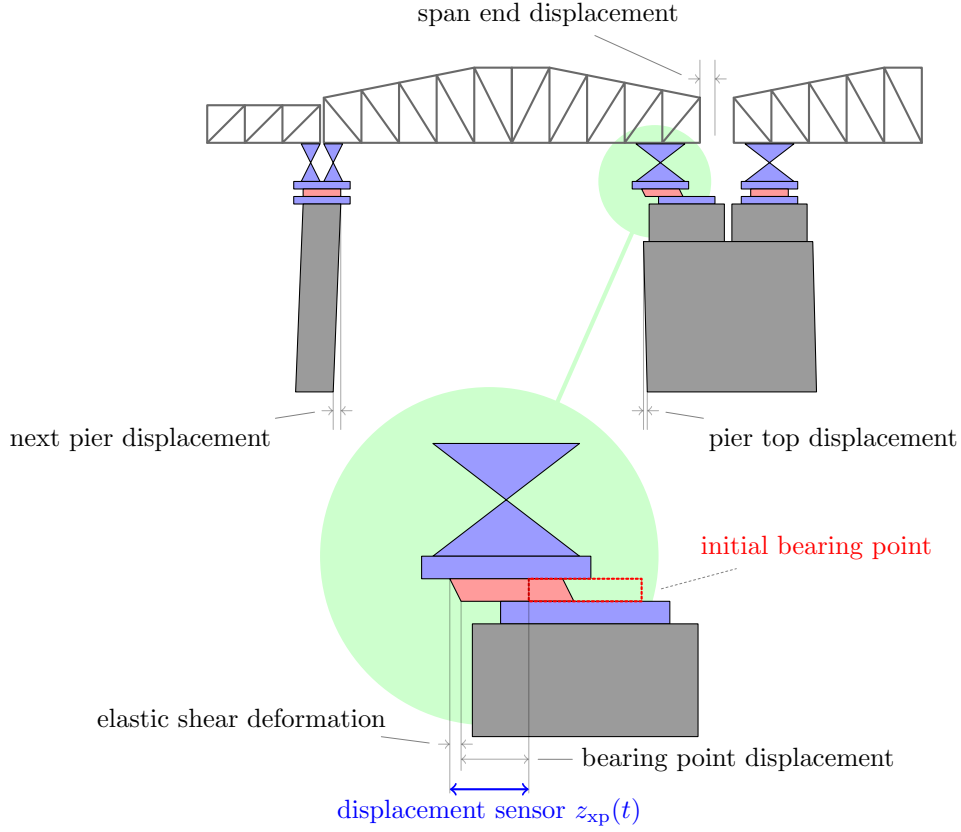

	\centering
	\includestandalone{figures/bridge}
	\caption{The structure consists of a pier (massive solid block) that is located between a bridge (horizontal structure) and its base. The pier is driven by a thermally induced displacement forcing and experiences dry friction (due to bridge/pier interaction). This mechanical structure is located at the parisian metro station Quai de la Gare (courtesy of RATP, the main public transportation in Paris)}
	\label{fig:bridge}
\end{figure}
\subsection{Presentation of the experimental data}
The model described above has been used on a practical case study on the Paris Metro Line 6 in 2016.
The aim of the operation was to estimate the real friction coefficients $f_s$ and $f_d$ of a single bearing point of the metro viaduct, near to the station Quai de la Gare in the East of Paris. The viaduct is an isostatic steel truss built in 1909. The bearing points are supposed to be fixed at one end and free at the other end of each span, but actually a significant friction appears on the free bearing points. OSMOS Group performed the continuous monitoring of the displacement of the span end on one of the free bearing points during one full year in 2015 and 2016. The measurements were taken 6 times every hour and additional records with 50 points every second were also available under the effects of the live loads. The bridge is instrumented with multiple sensors which measure 
(a) $T_\text{xp}$ the temperature in Celsius degree 
and (b) $z_\text{xp}$ the displacement sensor in millimeters.
The period of data assimilation covers about $\sim 300$ days (August 2015 - May 2016).
See Figure~\ref{fig:res-temporal}.
It is important to emphasize that the output of the displacement sensor $z_\text{xp}(t)$ is actually the sum of the bearing point displacement and the elastic shear deformation, as shown in Figure~\ref{fig:bridge}. The bearing point has a linear behavior, and its stiffness is known $K_\text{BP} = \SI{2.0e6}{\newton\per\meter}$. 
For the sake of consistency with the experimental data, in our numerical results we also use a corrected variable $z(t)$ that takes the elastic shear deformation into account $z(t) = x(t) + F(t) / K_\text{BP}$.

\subsection{Calibration of the model to interpret the experimental data}
\label{sec:calibration}

The experimental data, discussed in the previous section, showed that the length of a dynamic phase is smaller than the acquisition time-step.
Moreover, since this time-step is sufficiently small, the temperature variation between two time steps is also relatively small.
Therefore, we consider that the changes of temperature during dynamic phases are not significant. 
Physically, that justifies the use of the quasi-static approximation in the model for interpreting the data. 
Now, the goal is to find a set of parameters for the model that gives a good fit between the numerical and the experimental results.
The parameters considered are: a constant displacement offset $z_0$, the stiffness\footnotemark $K$, the dilatation coefficient $\beta$ and the two friction coefficients $f_s$ and $f_d$.
It is a difficult task to calibrate the model but on can use a least-square method and minimize the function:
\begin{equation}
	\textup{err}(z_0, K, \beta, f, f_0) = \int \left(z_0 +  z^{K, \beta, f, f_0}(T_\text{xp}(s)) - z_\text{xp}(s) \right)^2 ds
\end{equation} 
where the time integration is done over the data assimilation period ($\sim 300$ days), $T_\text{xp}(s)$ and $z_\text{xp}(s)$ are the experimental data for the temperature and the displacement respectively, and $z^{K, \beta, f, f_0}(T_\text{xp}(s))$ is the displacement given by the model detailed in section~\ref{sec:theory} with the quasistatic approximation, with $K, \beta, f_d, f_s$ as parameters and with $T_\text{xp}(s)$ as temperature input.
We use the genetic algorithm of the MATLAB optimization toolbox~\cite{matlabOTB} to look for the minimum of the function $\textup{err}$.
We have performed 10 tests with different arbitrary initial sets of parameters, we obtain the values presented in Table~\ref{tab:results}.
\begin{table}[h]
	\centering
	\begin{filecontents*}{figures/parameters.csv}
run, x0, K, f, f0, beta
1, 0.0095382,    0.17453e6, 0.41992, 0.47462, 0.00029736
2,  0.0091838,    0.20607e9, 0.41533, 0.47462, 0.00027181
3,  0.0098410,    0.18124e9, 0.41745, 0.47482, 0.00031842
4,  0.0097821,    0.19346e9, 0.43382, 0.49059, 0.00031350
5,  0.0097701,    0.19756e9, 0.43969, 0.46558, 0.00031369
6,  0.0097815,    0.18528e9, 0.41534, 0.46972, 0.00031345
7,  0.0097384,    0.18025e9, 0.40085, 0.45344, 0.00031218
8,  0.0097214,    0.18316e9, 0.41837, 0.45555, 0.00030986
9,  0.0097144,    0.18996e9, 0.42831, 0.47552, 0.00030753
10, 0.0096728,    0.19690e9, 0.42990, 0.45492, 0.00030618
Mean, 0.00967437, 0.188841e9,0.421898,0.468938,0.000306398
St. D.,  0.00019065, 0.009670e9,0.011200,0.011733,0.000013427
C. V.,  0.0197067, 0.051207,0.0265472,0.025020,0.043822

\end{filecontents*}
\pgfplotstabletypeset[
		col sep=comma,
		columns/run/.style={
			column name={run \#},
        	string type,
        },
		columns/x0/.style={
			column name={$z_0$ [\si{\meter}]},
        },
		columns/K/.style={
			column name={$K$ [\si{\newton\per\meter}]},
        },
		columns/f/.style={
			column name={$f_d$ [\si{\newton}]},
        },
		columns/f0/.style={
			column name={$f_s$ [\si{\newton}]},
        },
		columns/beta/.style={
			column name={$\beta$ [\si{\meter\per\kelvin}]},
        },
	    every head row/.style={before row=\toprule,after row=\midrule},
	    every last row/.style={before row=\midrule, after row=\bottomrule},
	    every row no 10/.style={before row=\midrule},
    ]{figures/parameters.csv}
	\caption{Results for 10 optimisation procedures with different random initial sets of values.}
	\label{tab:results}
\end{table} 
\subsection{Comparison between the experimental observations and our calibrated model}
Using the parameters shown in the first row of Table~\ref{tab:results}, we have obtained numerical results of our calibrated model in response to the observed thermal forcing.
These results are compared to the experimental observations in Figure~\ref{fig:res-temporal} and Figure~\ref{fig:res-phase}.
Figure~\ref{fig:res-temporal} presents the results over time; the main plot displays them over a 9 months period and the two other plots display a zoomed in result over a 2 month period.
Each of those plots displays on the top part both the measured and computed displacements, and on the bottom part the measured temperature.
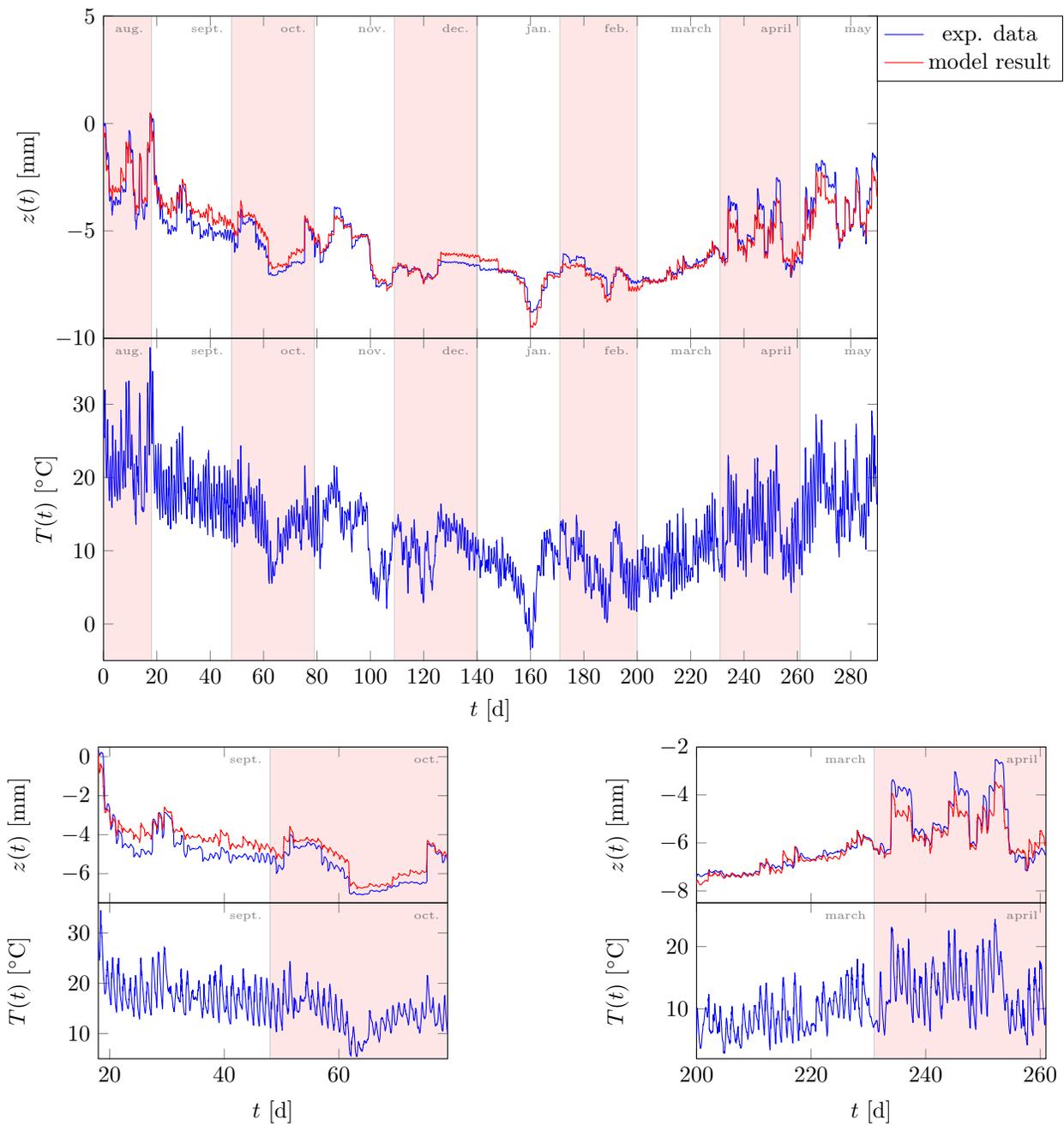
\begin{figure}[h]
	\centering
	    \tikzset{
	    every picture/.append style={
	        line join = round,
	        line cap = round,
	    },
    }
    \begin{tikzpicture}
		\pgfplotsset{
    		every axis/.append style={
	            axis background/.style={preaction={path picture={%
		            \draw [opacity = .1, fill=red] (axis cs:0,\pgfkeysvalueof{/pgfplots/ymin}) rectangle (axis cs:18,\pgfkeysvalueof{/pgfplots/ymax}) node[below left, opacity = .5] {\tiny \vphantom{flgp} aug.} ;%
		            \draw [opacity = .1, fill=white] (axis cs:18,\pgfkeysvalueof{/pgfplots/ymin}) rectangle (axis cs:48,\pgfkeysvalueof{/pgfplots/ymax}) node[below left, opacity = .5] {\tiny \vphantom{flgp} sept.} ;%
		            \draw [opacity = .1, fill=red] (axis cs:48,\pgfkeysvalueof{/pgfplots/ymin}) rectangle (axis cs:79,\pgfkeysvalueof{/pgfplots/ymax}) node[below left, opacity = .5] {\tiny \vphantom{flgp} oct.} ;%
		            \draw [opacity = .1, fill=white] (axis cs:79,\pgfkeysvalueof{/pgfplots/ymin}) rectangle (axis cs:109,\pgfkeysvalueof{/pgfplots/ymax}) node[below left, opacity = .5] {\tiny \vphantom{flgp} nov.} ;%
		            \draw [opacity = .1, fill=red] (axis cs:109,\pgfkeysvalueof{/pgfplots/ymin}) rectangle (axis cs:140,\pgfkeysvalueof{/pgfplots/ymax}) node[below left, opacity = .5] {\tiny \vphantom{flgp} dec.} ;%
		            \draw [opacity = .1, thick] (axis cs:140,\pgfkeysvalueof{/pgfplots/ymin}) -- (axis cs:140,\pgfkeysvalueof{/pgfplots/ymax}) ;%
		            \draw [opacity = .1, fill=white] (axis cs:140,\pgfkeysvalueof{/pgfplots/ymin}) rectangle (axis cs:171,\pgfkeysvalueof{/pgfplots/ymax}) node[below left, opacity = .5] {\tiny \vphantom{flgp} jan.} ;%
		            \draw [opacity = .1, fill=red] (axis cs:171,\pgfkeysvalueof{/pgfplots/ymin}) rectangle (axis cs:200,\pgfkeysvalueof{/pgfplots/ymax}) node[below left, opacity = .5] {\tiny \vphantom{flgp} feb.} ;%
		            \draw [opacity = .1, fill=white] (axis cs:200,\pgfkeysvalueof{/pgfplots/ymin}) rectangle (axis cs:231,\pgfkeysvalueof{/pgfplots/ymax}) node[below left, opacity = .5] {\tiny \vphantom{flgp} march} ;%
		            \draw [opacity = .1, fill=red] (axis cs:231,\pgfkeysvalueof{/pgfplots/ymin}) rectangle (axis cs:261,\pgfkeysvalueof{/pgfplots/ymax}) node[below left, opacity = .5] {\tiny \vphantom{flgp} april} ;%
		            \draw [opacity = .1, fill=white] (axis cs:261,\pgfkeysvalueof{/pgfplots/ymin}) rectangle (axis cs:291,\pgfkeysvalueof{/pgfplots/ymax}) node[below left, opacity = .5] {\tiny \vphantom{flgp} may} ;%
		            \draw [opacity = .1, fill=red] (axis cs:291,\pgfkeysvalueof{/pgfplots/ymin}) rectangle (axis cs:316,\pgfkeysvalueof{/pgfplots/ymax}) node[below left, opacity = .5] {\tiny \vphantom{flgp} june} ;%
		        }}},
	        	xmin = 0,
		        xmax = 290,
	        	restrict x to domain=0:290,
	        	filter discard warning=false,
    		},
		}
		\begin{axis}[
			height = 5cm,
			width = 12cm,
			name=displacement,
			%
	        xticklabel=\empty,
			ylabel = {$z(t)$},
	        y unit = {\si{\milli\meter}},
	        ymin = -10,
	        ymax = 5,
            legend style={at={(1,1)},anchor=north west,xshift=-.5\pgflinewidth, yshift=.5\pgflinewidth},
		]
	        \addplot+[mark=none] table [each nth point=6, x expr={\coordindex/24/6}, y=xp_disp, col sep=comma, unbounded coords=discard, search path={.,..,figures}] {results.csv} ;
	        \addlegendentry{exp. data}
	        \addplot+[mark=none] table [each nth point=6, x expr={\coordindex/24/6}, y=model_disp, col sep=comma, unbounded coords=discard, search path={.,..,figures}] {results.csv} ;
			\addlegendentry{model result}
		\end{axis}
		\begin{axis}[
			height = 5cm,
			width = 12cm,
			name=temperature,
			at = (displacement.south),
			anchor = north,
			xlabel = {$t$},
	        x unit = {\si{\day}},
			ylabel = {$T(t)$},
	        y unit = {\si{\degree C}},
	        ymin = -5,
	        ymax = 39,
            legend style={at={(1,1)},anchor=north west,xshift=-.5\pgflinewidth, yshift=.5\pgflinewidth},
		]
	        \addplot+[mark=none] table [each nth point=6, x expr={\coordindex/24/6}, y=temperature, col sep=comma, unbounded coords=discard, search path={.,..,figures}] {results.csv} ;

		\end{axis}
	\end{tikzpicture}
	
	\includestandalone{figures/temporal_z1}
	\hfill
	\includestandalone{figures/temporal_z2}
	\caption{
		Experimental and numerical results according to time.
		Main : 1 data point per hour.
		Zooms : 1 data point per 10 minutes.
	}
	\label{fig:res-temporal}
\end{figure}
Figure~\ref{fig:res-phase} displays both the experimental and numerical results in the \emph{displacement versus temperature} phase plan.
The loops predicted by the model have a similar width and height than those obtained with the experimental data.
A good agreement between experimental and numerical data is shown on these figures. That supports that the model is quite accurate.
This practical case study enables to validate the friction model and to identify the friction coefficients of the bearing point. 
\begin{figure}[h]
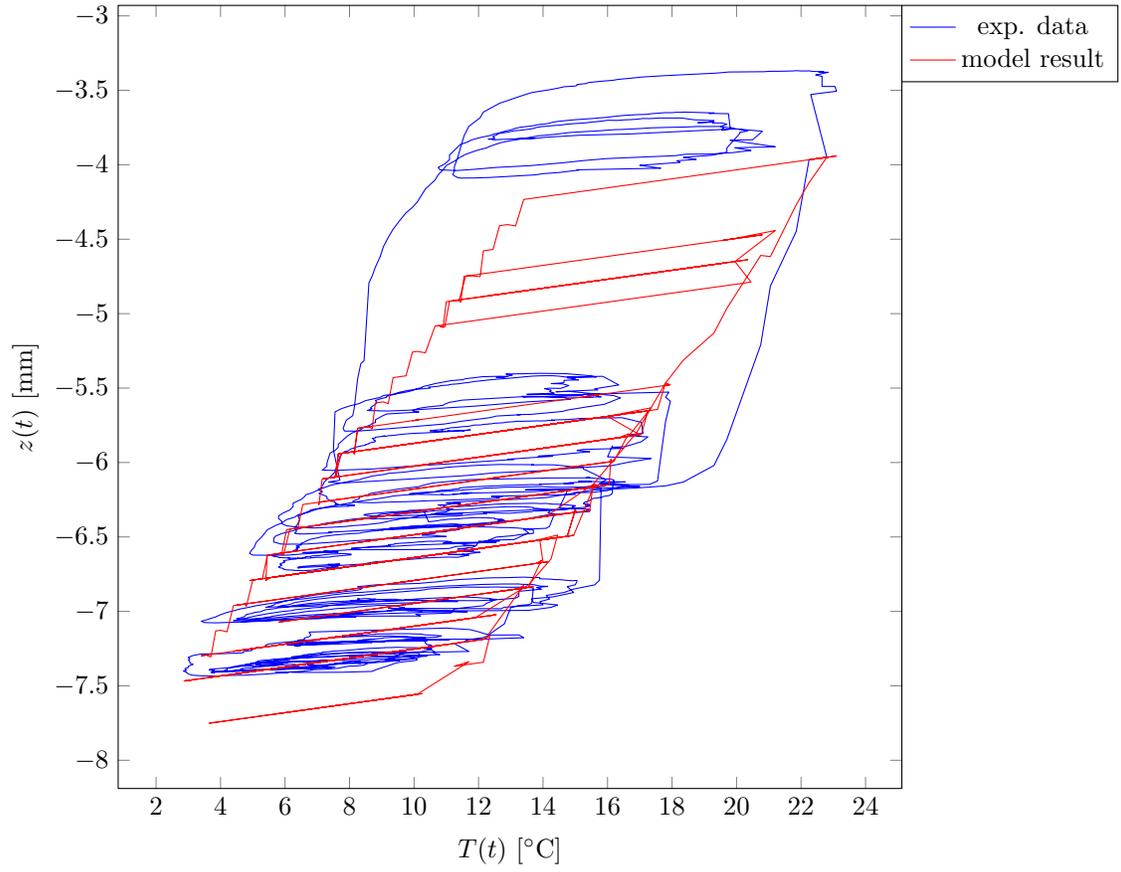

	\centering
	\includestandalone{figures/phase}

	\caption{
		Experimental and numerical results in the \emph{displacement versus temperature} phase plan.
		For the convenience of the reader, we only display the data for March and April. 
		1 data point per 10 minutes
	}
	\label{fig:res-phase}
\end{figure}

\section{Acknowledgement}
Alain Bensoussan research is supported by the National Science Foundation under grants DMS-1612880, DMS-1905449 and grant from the SAR Hong Kong RGC GRF 11303316. 
Laurent Mertz is supported by the National Natural Science Foundation of
China, Young Scientist Program \#11601335.
The authors wish to thank RATP Infrastructures for its support in the publication of the results on the Paris Metro Line 6 case.

\end{document}